\documentclass{amsart}
\usepackage{graphicx}
\usepackage{latexsym}
\usepackage{amsfonts}
\usepackage[all]{xy}
\usepackage{amssymb, mathrsfs, amsfonts, amsmath}
\usepackage{amsbsy}
\usepackage{amsfonts}

\usepackage{amssymb}
\newtheorem{theorem}{Theorem}[section]
\newtheorem{lemma}[theorem]{Lemma}
\newtheorem{cor}[theorem]{Corollary}
\newtheorem{proposition}[theorem]{Proposition}

\theoremstyle{definition}
\newtheorem{definition}[theorem]{Definition}

\newtheorem{question}[theorem]{Question}

\theoremstyle{remark}
\newtheorem{remark}[theorem]{Remark}

\numberwithin{equation}{section}
\newcommand{\OM}{\Omega}
\newcommand{\diver}{{{\rm{div}\,}}}

\newcommand{\hess}{\textrm{Hess}}
\newcommand{\grad}{\textrm{grad}\,}

\newcommand{\inj}{{\rm inj}}

\begin{document}

\title{Eigenvalue Estimates for  submanifolds  of  $N \times \mathbb{R}$\\ with locally bounded mean curvature}

\author{G. Pacelli Bessa}
\address{The Abdus Salam International Centre for Theoretical Physics, 34014 Trieste, Italy}
\curraddr{Department of Mathematics, Universidade Federal do Ceara-UFC, Campus do Pici, 60455-760 Fortaleza-CE
Brazil}
\email{bessa@mat.ufc.br}
\thanks{The first author was partially supported  by  a CNPq-grant and  ICTP Associate Schemes.}

\author{M. Silvana Costa}
\address{Department of Engineering, Universidade Federal do Ceara-UFC, Campus Cariri, Av. Castelo Branco, 150,
60030-200 Juazeiro do Norte-CE, Brazil}
\email{silvana\_math@yahoo.com.br}
\thanks{The second author was partially supported by a CNPq-scholarship}

\subjclass[2000]{Primary 53C40,
53C42; Secondary 58C40}

\date{ }

\dedicatory{}

\keywords{Fundamental tone estimates, minimal submanifolds, submanifolds with locally bounded mean curvature in $N\times \mathbb{R}$.}

\begin{abstract}
We give lower bounds for the fundamental tone of open sets in  submanifolds with locally bounded mean curvature in  $ N \times
\mathbb{R}$, where $N$ is an $n$-dimensional complete Riemannian
manifold with  radial sectional curvature $K_{N} \leq
\kappa$. When the immersion is minimal our estimates are sharp. We also show that cylindrically bounded minimal surfaces has positive fundamental tone.
\end{abstract}

\maketitle

\section{Introduction} The fundamental tone
$\lambda^{\ast}(\OM)$ of an open set $\OM$ in  a smooth Riemannian manifold $M$ is
defined by \[\lambda^{\ast}(\OM)=\inf\{\frac{\smallint_{\OM}\vert
\grad f \vert^{2}}{\smallint_{\OM} f^{2}};\, f\in
 H_{0}^1(\Omega) \backslash \{ 0\}\}.\]
 When $\OM=M$ is an open   Riemannian manifold,  the
fundamental tone $\lambda^{\ast}(M)$ coincides with the greatest
lower bound  $\inf \Sigma$ of the spectrum $\Sigma \subset
[0,\infty)$ of the unique self-adjoint extension of the Laplacian
$\triangle$ acting on $C_{0}^{\infty}(M)$ also denoted by
$\triangle$. When $\OM$ is compact with piecewise smooth boundary
$\partial \OM$ (possibly empty) then $\lambda^{\ast}(\OM)$ is the
first eigenvalue $\lambda_{1}(\OM)$ of $\OM$ (Dirichlet boundary data
if $\partial \OM\neq \emptyset$). A well studied problem in the geometry of the Laplacian is the relations between  the first eigenvalue or the fundamental tone of  open sets of Riemannian manifolds and the manifold's geometric invariants, see  \cite{berard},\cite{berger-gauduchon-mazet}, \cite{chavel} and references therein.  Another kind of problem is to give bounds for the the first eigenvalue (fundamental tone) of  open sets of minimal  submanifolds of Riemannian manifolds, see \cite{bessa-montenegro1}, \cite{bessa-montenegro2}, \cite{candel}, \cite{cheng-li-yau}, \cite{cheung-leung}.
 There has been recently an increasingly  interest on the study of  minimal surfaces (constant mean curvature)  in product spaces $N\times \mathbb{R}$, with  after the discovery of many beautiful examples in those spaces,  see \cite{meeks-rosenberg}, \cite{meeks-rosenberg2}. This motivates us to study the fundamental tone of  minimal submanifold of product spaces $N\times \mathbb{R}$. Our first result is the following theorem.
  \begin{theorem}\label{thm2}Let $\varphi :M\hookrightarrow N\times \mathbb{R} $ be a complete minimal $m$-dimensional submanifold, where $N$ has radial sectional curvature $K(\gamma (t))(\gamma'(t), v)\leq \kappa$, $v\in T_{\gamma (t)}N$, $\vert v\vert =1$, $v\perp \partial t$,  along the geodesics $\gamma (t)$ issuing from  a point  $x_{0}\in N$. Let $\OM \subset \varphi^{-1}(B_{N}(x_{0},r)\times \mathbb{R})$ be a connected component, where $r < \min \{\inj(x_{0}), \pi/2
\sqrt{\kappa} \}$, $(\pi/2
\sqrt{\kappa}=\infty$ if $\kappa\leq 0)$. Then \begin{equation}\label{sil-1}
\lambda^*(\Omega) \geq \lambda_1(B_{\mathbb{N}^{m-1}(\kappa)}(r)).
  \end{equation}If  $\Omega$ is bounded then inequality (\ref{sil-1}) is strict. Here $\mathbb{N}^{m-1}(\kappa)$ is the  $(m-1)$-dimensional simply connected space form of constant sectional curvature $\kappa$.
  \end{theorem}Theorem (\ref{thm2}) can be viewed as a version of Theorem (1.10) of \cite{bessa-montenegro2} for product spaces. There Bessa and Montenegro gave eigenvalue estimates for pre-images of geodesic balls in Riemannian manifolds with radial sectional curvature bounded above, here we give lower eigenvalue estimates for pre-images of cylinders in product spaces.
  \begin{remark}Inequality (\ref{sil-1}) is sharp. For if we let  $\varphi : \mathbb{H}^{m-1}(-1)\times \mathbb{R}\hookrightarrow \mathbb{H}^{n}(-1)\times \mathbb{R}$ be given by $\varphi (x, t)=(i(x),t)$, where  $i:\mathbb{H}^{m-1}(-1)\hookrightarrow \mathbb{H}^{n}(-1)$ is a totally geodesic embedding then  for $\Omega=\varphi^{-1}(B_{\mathbb{H}^{n}(-1)}(r)\times \mathbb{R})=B_{\mathbb{H}^{m-1}(-1)}(r)\times \mathbb{R}$ we have \[ \lambda^*(\Omega)= \lambda_1(B_{\mathbb{H}(-1)^{m-1}}(r)).\]
\end{remark}
\begin{cor}Let $\varphi :M\hookrightarrow   \mathbb{R}^{3}$ be a complete minimal surface with $\varphi (M)\subset B_{\mathbb{R}^{2}}(r)\times \mathbb{R}$. Then \begin{equation}\lambda^{\ast}(M)\geq\lambda_{1}(B_{\mathbb{R}^{2}}(r)= \frac{c_{0}}{r^{2}},\end{equation}where $c_{0}$ is the first zero of the $J_{0}$-Bessel function.
\end{cor}\begin{question} The only examples known (to the best of our knowledge) of complete surfaces in $\mathbb{R}^{3}$  with positive fundamental tone  are the Nadirashvilli bounded minimal surfaces \cite{Nadirashvili} and Martin-Morales cylindrically bounded minimal surfaces \cite{pacomartin}. Both  Nadirashvilli and Martin-Morales minimal surfaces have at least two bounded coordinate functions.  That was crucial in the proof that their fundamental tones were positive.
 That raises the question whether  there are minimal surfaces in $\mathbb{R}^{3}$ with at most one bounded coordinate function with positive fundamental tone. More specifically, has the Jorge-Xavier minimal surface inside the slab \cite{jorge-xavier}  positive fundamental tone?
\end{question}

A second purpose of this paper is study  the fundamental tones of domains in submanifolds of $N\times \mathbb{R}$ with locally bounded  mean curvature. We need a stronger notion of {\em locally bounded mean curvature} than the considered  in \cite{bessa-montenegro1}.
\begin{definition}An isometric immersion $\varphi : M \hookrightarrow W\times \mathbb{R}$ has locally bounded mean curvature $\vert H\vert$
if  for any $p\in W$ and $r>0$, the number
\[h(p,r)=\sup \{\vert H(x)\vert ;\,x\in \varphi (M)\cap (B_{W}(p,r)\times \mathbb{R}) \} \] is finite. Here  $B_{W}(x_{0},r)$ is the geodesic ball of radius $r$ and center $x_{0}$ in $W$.\label{defLB}
\end{definition}
Our second result is  the following  theorem.
\begin{theorem}\label{thm3}Let $\varphi : M \hookrightarrow N\times \mathbb{R}$ be  a complete  immersed $m$-submanifold of $N\times \mathbb{R}$ with locally bounded  mean curvature, where $N$ has radial sectional curvature $K\leq \kappa$,   along the geodesics issuing from  a point  $x_{0}\in N$. Let $\Omega(r)\subset \varphi^{-1}( B_{N}(p,r)\times \mathbb{R})$ be
any connected component. Suppose $r\leq \min\{\inj_{N}(x_{0}), \pi/2\sqrt{\kappa}\}$. In addition \begin{itemize}\item  If $\vert h(x_{0},r)\vert <\Lambda^{2}<\infty$ let
$r\leq  (C_{\kappa}/S_{\kappa})^{-1}(\Lambda^{2}/(m-2)). $ \item If
 $\lim_{r\to \infty}h(x_{0},r)=\infty$ let
  $r\leq  (C_{\kappa}/S_{\kappa})^{-1}(h(x_{0},r_{0})/(m-2)),  $ where $r_{0}$  is so that $(m-2)\frac{C_{\kappa}}{S_{\kappa}}(r_{0})-h(x_{0}, r_{0})=0$.\end{itemize} In both cases we have
    \[\lambda^{\ast}(\OM(r))\geq \left[\frac{(m-2)\displaystyle\frac{C_{\kappa}}{S_{\kappa}}(r)-h(x_{0}, r)}{2}\right]^{2}>0\cdot\]
\end{theorem}
\begin{cor}[Bessa-Montenegro, \cite{bessa-montenegro3}]Let $\varphi : M \hookrightarrow N\times \mathbb{R}$ be  a compact  immersed submanifold with mean curvature vector $H$. Let $p_{1}:N\times \mathbb{R}\to N$ be the projection on the first factor. Then the extrinsic radius  of $p_{1}(M)$ is given  \[R_{p_{1}(M)}=(\frac{C_{\kappa}}{S_{\kappa}})^{-1}(\sup_{M} \vert H\vert/(m-2)) .\]\label{cor1}
\end{cor}

\section{ Preliminaries}

Let $\varphi : M \hookrightarrow W$ be an isometric immersion, where
$M$ and $W$ are complete Riemannian manifolds of dimension
 $m$ and $n$  respectively. Consider a smooth function $g:W \rightarrow \mathbb{R}$ and the composition $f=g\,\circ\,
 \varphi :M \rightarrow \mathbb{R}$. Identifying $X$ with $d\varphi (X)$ we have at $q\in M$ that
 the Hessian of $f$ is given by
 \begin{equation}\hess\,f (q)  \,(X,Y)= \hess\,g (\varphi (q))\,(X,Y) +
 \langle \grad\,g\,,\,\alpha (X,Y)\rangle_{\varphi (q)}.
\label{eqBF2}
\end{equation}
 Taking the trace in (\ref{eqBF2}), with respect to an orthonormal basis $\{ e_{1},\ldots e_{m}\}$
 for $T_{q}M$, we have the Laplacian of $f$,
\begin{eqnarray}
\Delta \,f (q) & = & \sum_{i=1}^{m}\hess\,g (\varphi
(q))\,(e_{i},e_{i}) + \langle \grad\,g\,,\, \sum_{i=1}^{m}\alpha
(e_{i},e_{i})\rangle.\label{eqBF3}
\end{eqnarray}
The formulas (\ref{eqBF2}) and (\ref{eqBF3}) are  well known in
the literature,
 see  \cite{jorge-koutrofiotis}.
\vspace{2mm}
\noindent For the proof of Theorems (\ref{thm2}) and (\ref{thm3}) we will need few  preliminaries results.
 The first result we need is the Hessian Comparison Theorem, one can see \cite{kn:s-y} for a proof.
\begin{theorem}[Hessian Comparison Theorem]Let $W$ be a
complete Riemannian manifold and let $\rho
$ be the  distance function on $W$  to $x_{0}$. Let
 $\gamma$ be a minimizing geodesic starting at $x_{0}$  and suppose that the radial sectional curvatures  of $M$ along $\gamma$ is bounded above $K_{\gamma}\leq \kappa$. Then  the Hessian  of  $\rho$ at $\gamma(t)$ satisfies

\begin{equation}
Hess\,\rho(\gamma (t))(X,X)\geq \displaystyle \frac{C_{\kappa}}{S_{\kappa}}(t)\cdot\Vert
X\Vert^{2},
\,\,\,\,\, X\perp \gamma'(t)
\label{eqBF6}
\end{equation}
Where  \label{thmHess}
\begin{equation}S_{\kappa}(t)=\left\{
\begin{array}{lll}\displaystyle \sin(\sqrt{\kappa} \cdot t)/\sqrt{\kappa} & if & \kappa >0\\
&&\\
\displaystyle1/t &if & \kappa =0\\
&&\\
\displaystyle\sinh(\sqrt{-\kappa}\cdot t)/\sqrt{-\kappa} &if & \kappa <0
\end{array}\right.
\end{equation}and  $C_{\kappa}(t)=S'_{\kappa}(t)$,
\end{theorem}
\noindent
The second and third results we need are eigenvalue estimates proved in  \cite{bessa-montenegro2} and in \cite{bessa-montenegro1}. The former is  a generalization of the well known Barta's eigenvalue theorem \cite{barta} and the later is a generalization of a result of Cheung-Leung \cite{cheung-leung}.
\begin{theorem}[\cite{bessa-montenegro2}]\label{thmBM1}Let
$\OM$ be an open set in  a Riemannian manifold $M$. Then
\begin{equation}\label{eqThmP1}
 \lambda^{\ast}(\OM)\geq
 \sup_{{\mathcal C}(\OM) }\{\inf_{ \OM} (\diver X-
\vert X \vert^{2})\}. \end{equation}Where ${\mathcal C}(\OM)$ is the set of smooth vector fields  in $\OM\setminus F$  for some  closed set  $F$
with Hausdorff measure $\mathcal{H}^{n-1}(F\cap \Omega)=0$.
\end{theorem}\noindent

 \begin{theorem}[\cite{bessa-montenegro1}]\label{thmBM2}Let
$\OM$ be an open set in  a Riemannian manifold $M$ and $c(\OM)$ a constant defined by \[c(\OM)=\sup_{ \mathcal{C}_{+}(\Omega)}\frac{(\inf_{\Omega} \diver X)}{\sup_{\Omega} \vert X\vert^{2}},\]where  $\mathcal{C}_{+}(\Omega)=\{X\in \mathcal{C}(\OM)\: \inf_{\OM}\diver X>0\, {\rm and}\, \sup_{\OM}\vert X\vert <\infty\}$. Then
\begin{equation}\label{eqThmP2}
 \lambda^{\ast}(\Omega )\geq \frac{c(\OM)^{2}}{4}
 \end{equation}
\end{theorem}
\noindent Finally, we need the following technical lemma.
\begin{lemma}\label{wolverine}
Let $v: B_{\mathbb{N}^n(\kappa)}(r) \rightarrow \mathbb{R}$ be a
first positive eigenfunction of  $B_{\mathbb{N}^n(\kappa)}(r)\subset \mathbb{N}^{n}(\kappa)$ associated to the first eigenvalue
$\lambda_1(B_{\mathbb{N}^n(\kappa)}(r))$. Then
\begin{equation}\label{wolverine2}
n\frac{C_{\kappa(t)}}{S_{\kappa}(t)} v'(t)+
\lambda_1(B_{\mathbb{N}^n(\kappa)}(r))v(t) < 0, \,\,\,t \in(0,r).
\end{equation}
\end{lemma}
\begin{proof}We are going to treat the cases $\kappa <0$, $\kappa =0$ and $\kappa>0$ separately. Suppose first that $\kappa <0$ and let us call $\lambda=\lambda_1(B_{\mathbb{N}^n(\kappa)}(r))$ for simplicity of notation. Recall that $v(t)$ satisfies the following differential equation.
\begin{equation} v''(t)+(n-1)\frac{C_{\kappa}}{S_{\kappa}}(t)v'(t)+\lambda v(t)=0, \,\,\,t \in(0,r)\label{eqSilvana1}\end{equation}
Consider the  function $ \mu(t)
=C_{\kappa}(t)^{\displaystyle\frac{\lambda}{n\kappa}} $. Thus
$
\mu'(t) = -\displaystyle\frac{\lambda}{n}\; S_{\kappa}(t)
\;C_{\kappa}(t)^{\displaystyle\frac{\lambda}{n\kappa} -1}
$
and
\begin{equation}\label{eqpacelli2}
\begin{array}{ccl}
v'(t)\mu(t) - \mu'(t)v(t) &=& v'(t)\;C_{\kappa}(t)^{\displaystyle\frac{\lambda}{nk}} + \displaystyle\frac{\lambda}{n} S_{\kappa}(t)\;C_{\kappa}(t)^{\displaystyle\frac{\lambda}{n\kappa}-1}v(t)\\
&&\\
&=&\displaystyle\frac{1}{n}\;C_{\kappa}(t)^{\displaystyle\frac{\lambda}{n\kappa}-1}\;S_{\kappa}(t)
\left( \displaystyle n\frac{C_{\kappa}(t)}{S_{\kappa}(t)}v'(t) + \lambda
v(t)\right)
\end{array}
\end{equation}From (\ref{eqpacelli2}) we see that to prove that $
n\displaystyle\frac{C_{\kappa}(t)}{S_{\kappa}(t)}v'(t) + \lambda v(t) <0
$ we only  need to prove
\[v'(t)\mu(t) - \mu'(t)v(t) < 0.\]

\vspace{2mm}
\noindent Multiplying the equation (\ref{eqSilvana1}) by $S_{\kappa}^{n-1}$,  we obtain the following differential equation
\begin{equation}\label{eqSilvana2}
\left(S_{\kappa}^{n-1}\;v'\right)'(t) +
\lambda S_{\kappa}^{n-1}(t)\;v(t)=0, \,\,\,t \in(0,r).
\end{equation}
The function $\mu(t) =
C_{\kappa}(t)^{\displaystyle\frac{\lambda}{n\kappa}}$, satisfies the differential equation
\begin{equation}\label{eqSilvana3}
\begin{array}{ccl}
\mu''(t)& =& - \displaystyle\lambda\left( \displaystyle\frac{1}{nC_{\kappa}^{2}(t)}
-\displaystyle\frac{\lambda}{n^2}\displaystyle\frac{S_{\kappa}^{2}(t)}{C_{\kappa}^{2}(t)}\right)
\mu(t).\end{array}
\end{equation}
Multiplying the equation (\ref{eqSilvana3}) by $S_{\kappa}^{n-1}(t)$ we  obtain
\begin{equation}
\begin{array}{ccl}
S_{\kappa}^{n-1}(t)\mu''(t) +  \displaystyle\lambda S_{\kappa}^{n-1}(t)\left(
\displaystyle\frac{1}{n\,C_{\kappa}^{2}(t)}
-\displaystyle\frac{\lambda}{n^2}\frac{S_{\kappa}^{2}(t)}{C_{\kappa}^{2}(t)}\right)
\mu(t) = 0.
\end{array}
\end{equation}
Adding and subtracting the term $(n-1)\mu'(t)S_{\kappa}^{n-2}(t)C_{\kappa}(t) $ we obtain
\begin{equation}\label{eqSilvana4}
\begin{array}{ccl}\displaystyle
(S_{\kappa}^{n-1}\mu')'(t) + \lambda S_{\kappa}^{n-1}(t)
\left(\displaystyle\frac{n-1}{n} +\displaystyle \frac{1}{nC_{\kappa}^{2}(t)}
-\displaystyle\frac{\lambda }{n^2}\displaystyle\frac{S_{\kappa}^{2}(t)}{C_{\kappa}^{2}(t)}\right)
\mu(t) = 0
\end{array}
\end{equation}
The functions $v$ and $\mu$ then satisfy the follow identities:
\begin{equation}\label{eqSilvana5}
\begin{array}{rll}
\left(S_{\kappa}^{n-1}\;v'\right)'(t) +
\lambda S_{\kappa}^{n-1}(t)\;v(t)&=&0\\
&&\\
(S_{\kappa}^{n-1}\mu')'(t) + \lambda S_{\kappa}^{n-1}(t)
\left(\displaystyle\frac{n-1}{n} + \displaystyle\frac{1}{n\,C_{\kappa}^{2}(t)}
-\displaystyle\frac{\lambda }{n^2}\displaystyle\frac{S_{\kappa}^{2}(t)}{C_{\kappa}^{2}(t)}\right)
\mu(t)& =& 0
\end{array}
\end{equation}
Multiply the first identity  of (\ref{eqSilvana5}) by $\mu(t)$ and the second identity by $-v(t)$ adding
them and integrating from $0$ to $t$ we obtain
\begin{equation}
\begin{array}{ccl}
S_{\kappa}^{n-1}\left(v'\mu - \mu'v\right)(t) =\displaystyle
-\int_{0}^{t}\lambda S_{\kappa}^{n-1}(t) \left(\displaystyle\frac{1}{n} -
\displaystyle\frac{1}{nC_{\kappa}^{2}(t)}
+\displaystyle\frac{\lambda }{n^2}\displaystyle\frac{S_{\kappa}^{2}(t)}{C_{\kappa}^{2}(t)}\right)
\mu(t)v(t) dt
\end{array}
\end{equation}
Clearly
\[
\displaystyle\lambda S_{\kappa}^{n-1}(t) \left(\displaystyle\frac{1}{n} -
\displaystyle\frac{1}{nC_{\kappa}^{2}(t)}
+\displaystyle\frac{\lambda }{n^2}\displaystyle\frac{S_{\kappa}^{2}(t)}{C_{\kappa}^{2}(t)}\right)
\mu(t)v(t) > 0\]
Therefore  we have
$
v'(t)\mu(t) - \mu'(t)v(t) < 0$ for $t \in (0,r).
$ This settle the case $\kappa<0$.

\vspace{2mm}
\noindent Suppose  that $\kappa > 0$. We have
$
S_{\kappa}(t) = \displaystyle\frac{1}{\sqrt{\kappa}}\sin{\sqrt{\kappa}t}
$
for $t \in (0,r)$ with $r < \displaystyle\frac{\pi}{\sqrt{\kappa}}$.
Define
$\mu(t) =
C_{\kappa}(t)^{\displaystyle\frac{-\lambda}{n\kappa}}$.
Thus $\mu'(t)=\displaystyle\frac{\lambda}{n}S_{\kappa}(t)C_{\kappa}(t)^{\displaystyle\frac{-\lambda}{n\kappa}-1}$. With a similar procedure we obtain that  $v$ and $
\mu$ satisfy the following differential identities
\begin{equation}
\begin{array}{rcl}
\left(S_{\kappa}^{n-1}\;v'\right)'(t) +
\lambda S_{\kappa}^{n-1}(t)\;v(t)&=&0\\
&&\\
(S_{\kappa}^{n-1}\mu')'(t) - \lambda S_{\kappa}^{n-1}(t)
\displaystyle\left(\frac{n-1}{n} + \displaystyle\frac{1}{nC_{\kappa}^{2}(t)}
+\displaystyle\frac{\lambda}{n^2}\displaystyle\frac{S_{\kappa}^{2}(t)}{C_{\kappa}^{2}(t)}\right)
\mu(t) &=& 0
\end{array}
\end{equation}
In (28) we multiply the first identity by $\mu$ and the second by
$-v$ adding them and integrating from $0$ to $t$ the resulting
identity we obtain
\begin{equation}
S_{\kappa}^{n-1}(v'\mu - \mu'(t)v)(t) =
-\int_{0}^{t}\lambda_1S_{\kappa}^{n-1}(t) \left( 2 -\displaystyle\frac{1}{n} +
\displaystyle\frac{1}{nC_{\kappa}^{2}(t)}
+\displaystyle\frac{\lambda}{n^2}\frac{S_{\kappa}^{2}(t)}{C_{\kappa}^{2}(t)}\right)
\mu(t)v(t) dt
\end{equation}
 The term $\displaystyle
\lambda S_{\kappa}^{n-1}(t) \left( 2 -\frac{1}{n} +
\frac{1}{nC_{\kappa}^{2}(t)}
+\frac{\lambda }{n^2}\frac{S_{\kappa}^{2}(t)}{C_{\kappa}^{2}(t)}\right)
\mu(t)v(t) > 0
$ is positive for $t\in (0,r)$, $r< \pi/2\sqrt{\kappa}$. Therefore
we have that
$
v'(t)\mu(t) - \mu'(t)v(t) < 0$ for $t \in (0,r)$, $r< \pi/2\sqrt{\kappa}$

\vspace{2mm}

\noindent The case $\kappa=0$ we proceed similarly. Define  $\mu (t)=e^{-\displaystyle\frac{\lambda t^{2}}{2n}}$.
The functions $v$ and $\mu$ satisfy the following identities,
\begin{equation}\label{eqSubm6}\begin{array}{lcl} (t^{n-1}v'(t))' +
\lambda t^{n-1}v(t)\,=\,0& & \\
&& \\
(t^{n-1}\mu'(t))' + \lambda t^{n-1}(1-
\displaystyle\frac{\lambda \,t^{2}}{n^{2}})\mu(t)\,=\,0& &
\end{array}
\end{equation}In (\ref{eqSubm6}) we multiply the first identity by
$\mu$ and the second by $-v$ adding them and integrating from $0$
to $t$ the resulting identity we obtain, $$t^{n-1}(v'(t)\,\mu(t)
-\,v(t)\,\mu'(t))=-\frac{\lambda^{2}}{n^{2}}\int_{0}^{t}\mu(t)\,v(t)
<0,\,\,\ \forall t\in (0,r).
$$
Then $\mu(t)v'(t) - \mu'(t)v(t)<0. $  This proves the lemma.
\end{proof}

\section{ Proof of the Results}
\subsection{Proof of Theorem \ref{thm2}}

\begin{proof}
Let $\varphi:M\hookrightarrow N\times \mathbb{R}$ be a minimal immersion of a $m$-dimensional Riemannian manifold  $M$, where $N$ is a  complete Riemannian $n$-manifold with radial  sectional curvature along the geodesics $\gamma (t)$ issuing from  a point  $x_{0}\in N$ bounded from above  $K(\gamma'(t), v)\leq \kappa$, $v\in T_{\gamma (t)}N$, $\vert v\vert =1$, $v\perp \partial t$ . Let  $\OM \subset \varphi^{-1}(B_{N}(x_{0},r)\times \mathbb{R})$, $r< \min \{\inj(x_{0}), \pi/2
\sqrt{\kappa} \}$, $(\pi/2
\sqrt{\kappa}=\infty$ if $\kappa\leq 0)$ be a connected component.  Let $\rho_{N}(x)={\rm dist}_{N}(x_{0},x)$ be the distance function  in $N$ to $x_{0}$
and let $v:
B_{\mathbb{N}^{m-1}(\kappa)}(r) \rightarrow \mathbb{R}$ be a first
positive eigenfunction associated with the first eigenvalue
$\lambda_1(B_{\mathbb{N}^{m-1}(\kappa)}(r)) $ of the geodesic ball of radius $r$ in the simply connected
 $(m-1)$-dimensional space form $\mathbb{N}^{m-1}(\kappa)$ of constant sectional curvature
$\kappa$. The eigenfunction $v$ is radial, i.e. $v(x)=v(\vert x\vert)$ and we can look at $v$ as it were defined in $[0,r]$   satisfying the equation
\begin{equation}\label{sil01}
v''(t) + (m-2)\frac{C_{\kappa}}{S_{\kappa}}(t) v'(t) +
\lambda_1(B_{\mathbb{N}^{m-1}(\kappa)}(r))v(t) = 0, \;\;\; t\in
(0,r).
\end{equation} Choose the first eigenfunction that satisfies the initial conditions
 $v(0) = 1$ and $v'(0)=0$.
Define $g:B_{N}(r)\times
\mathbb{R}\rightarrow \mathbb{R}$ by $g= v\circ
\rho_{N}\circ p$ and  $f:\Omega
\rightarrow \mathbb{R}$ by  $f= g \circ \varphi$,  where $p:N\times \mathbb{R} \rightarrow N$ is
 the projection in the first factor.  Setting $X=\grad \log f$ we have that $\diver X-\vert X\vert^{2}=\triangle f/f$.
 Thus by Theorem   (\ref{thmBM1}) we have  that \[\lambda^{\ast}(\OM)\geq \inf_{\OM}(-\displaystyle\frac{\triangle f}{f}).\]
\noindent We are going to give lower bound $-\triangle f/f$.  Let $x\in \OM$  and
 $ \{e_{1},\ldots, e_{m}\}$ be any orthonormal basis
for $T_{x}\OM$. The Laplacian of $f$ at $x$ is given by
\begin{equation}\label{eqLaplacian}
\triangle_{M}\,f (x)    =   \sum_{i=1}^{m}\hess_{(N\times \mathbb{R})}\,g (\varphi
(x))\,(e_{i},e_{i})=\sum_{i=1}^{m}\hess_{N}\,v\circ \rho_{N}(q)(e_{i}, e_{i})\end{equation}
\noindent Consider the orthonormal basis $\{\grad \rho_{N}, \partial/\partial\theta_{1}, \ldots, \partial/\partial\theta_{n-1},\,\partial/\partial s \}$ for $T_{(q,s)}(N\times \mathbb{R})$, where $\{ \grad \rho_{N}, \partial/\partial\theta_{1}, \ldots, \partial/\partial\theta_{n-1}\}$ is an orthonormal basis for $T_{q}N$ (polar
coordinates).
Let  $\{e_{1}, \ldots, e_{m}\}$ be an orthonormal basis for $T_{x}\OM$ and write\begin{equation}\label{eqBase}e_{i}=a_{i}\cdot\grad \rho_{N}+b_{i}\cdot\partial/\partial s +\sum_{j=1}^{n-1}c_{i}^{j}\cdot\partial/\partial\theta_{j}.\end{equation}  Where $a_{i}, b_{i}, c_{i}^{j}$ are constants satisfying $ a_{i}^{2}+b_{i}^{2}+\sum_{j=1}^{n-1}(c_{i}^{j})^{2}=1, \,i=1,\ldots,m $. Computing $\triangle_{M} f (x)$ we have, (recall that $\varphi (x)=(q,s)$ and we are letting $t=\rho_{N}(q)$)
\begin{eqnarray} \triangle f (x)&=& \sum_{i=1}^{m}\left[e_{i}(v'(t))\langle \grad \rho_{N}, e_{i}\rangle + v'(t)\hess_{N}\rho_{N}(e_{i},e_{i})\right]\nonumber \\
&=&v''(t)\sum_{i=1}^{m}a_{i}^{2}+v'(t)\sum_{i=1}^{m}\sum_{j=1}^{n-1}(c_{i}^{j})^{2}\hess \rho_{N}(\partial/\partial\theta_{j},\partial/\partial\theta_{j})\nonumber
\end{eqnarray}Since $v'(t)\leq 0$ we have by the Hessian Comparison Theorem that
\begin{eqnarray}\label{eqSilvana9} -\triangle f (x)&\geq & -v''(t)\sum_{i=1}^{m}a_{i}^{2}-
v'(t)\frac{C_{\kappa}}{S_{\kappa}}(t)\sum_{i=1}^{m}\sum_{j=1}^{n-1}(c_{i}^{j})^{2}\nonumber\\ && \nonumber  \\
&=& -v''(t)\sum_{i=1}^{m}a_{i}^{2}-
v'(t)\frac{C_{\kappa}}{S_{\kappa}}(t)\left[m-\sum_{i=1}^{m}a_{i}^{2}-\sum_{i=1}^{m}b_{i}^{2}\right]\nonumber\\
&&\nonumber \\
&=&-v''(t)-(m-2)v'(t)\frac{C_{\kappa}}{S_{\kappa}}(t) \\
&&\nonumber  \\
&& +v''(t)\left[1-\sum_{i=1}^{m}a_{i}^{2}\right]
-v'(t)\frac{C_{\kappa}}{S_{\kappa}}(t)\left[1-\sum_{i=1}^{m}a_{i}^{2}+1-\sum_{i=1}^{m}b_{i}^{2}\right]\nonumber \\ &&\nonumber \\
&=& \lambda_{1}(B_{\mathbb{N}^{m-1}(\kappa)}(r))v(t)\nonumber \\
&& \nonumber \\ && +v''(t)\left[1-\sum_{i=1}^{m}a_{i}^{2}\right]
-v'(t)\frac{C_{\kappa}}{S_{\kappa}}(t)\left[1-\sum_{i=1}^{m}a_{i}^{2}+1-\sum_{i=1}^{m}b_{i}^{2}\right]\nonumber
\end{eqnarray}
\noindent We will show that the last line of (\ref{eqSilvana9}) is nonnegative, this is  \begin{equation} \label{eqSilvana10} v''(t)\left[1-\sum_{i=1}^{m}a_{i}^{2}\right]
-v'(t)\displaystyle\frac{C_{\kappa}}{S_{\kappa}}(t)\left[1-\sum_{i=1}^{m}a_{i}^{2}+1-\sum_{i=1}^{m}b_{i}^{2}\right]\geq 0.\end{equation}

\noindent Substituting $v''(t)=-(m-2)v'(t)\displaystyle\frac{C_{\kappa}}{S_{\kappa}}(t)-\lambda_{1}(B_{\mathbb{N}^{m-1}(\kappa)}(r))v(t)$ in  (\ref{eqSilvana10}) we obtain
 \begin{eqnarray}\label{eqIgualdade}v''(t)\left[1-\sum_{i=1}^{m}a_{i}^{2}\right]
-v'(t)\displaystyle\frac{C_{\kappa}}{S_{\kappa}}(t)\left[1-\sum_{i=1}^{m}a_{i}^{2}+1-\sum_{i=1}^{m}b_{i}^{2}\right]&=& \nonumber \\
&&\nonumber \\
-\left[(m-1)v'(t)\displaystyle\frac{C_{\kappa}}{S_{\kappa}}(t)+
\lambda_{1}(B_{\mathbb{N}^{m-1}(\kappa)}(r))v(t)\right]\left[1-\sum_{i=1}^{m}a_{i}^{2}\right]&& \\
&&\nonumber \\
-v'(t)\displaystyle\frac{C_{\kappa}}{S_{\kappa}}(t)\left[1-\sum_{i=1}^{m}b_{i}^{2}\right]&\geq &0\nonumber
 \end{eqnarray}since  we have that $ (m-1)v'(t)\displaystyle\frac{C_{\kappa}}{S_{\kappa}}(t)+
\lambda_{1}(B_{\mathbb{N}^{m-1}(\kappa)}(r))v(t)<0$ by  Lemma (\ref{wolverine}) and $\left[1-\sum_{i=1}^{m}a_{i}^{2}\right]\geq 0$ and $\left[1-\sum_{i=1}^{m}b_{i}^{2}\right]\geq 0$.
From (\ref{eqSilvana9}) we have $\displaystyle -\frac{\triangle f }{f}(x)\geq \lambda_{1}(B_{\mathbb{N}^{m-1}(\kappa)}(r)).$  Therefore, \[\lambda^{\ast}(\OM)\geq \inf_{\OM}(-\triangle f /f)\geq \lambda_{1}(B_{\mathbb{N}^{m-1}(\kappa)}(r)).\]

   \noindent To prove the last assertion of Theorem  (\ref{thm2}) we need the following proposition proved in \cite{bessa-montenegro2}.

  \begin{proposition}\label{2.2}Let $\Omega$ be a bounded
domain in a smooth Riemannian manifold. Let $v\in C^{2}(\Omega)\cap
C^{0}(\overline{\Omega})$, $v>0$ in $\Omega$ and $v\vert
\partial \Omega =0$. Then \begin{equation}\lambda^{\ast}(\Omega)\geq
\inf_{\Omega}(-\frac{\triangle
v}{v}).\label{eqProp1}\end{equation}Moreover,
$\lambda^{\ast}(\Omega)= \inf_{\Omega}(-\displaystyle\frac{\triangle
v}{v})$ if and only if $v=u$,  where $u$ is a positive eigenfunction
of $\Omega$, i.e. $\triangle u+\lambda^{\ast}(\Omega)u=0$.
\end{proposition}\noindent If we have equality $\lambda_{1}(\OM)= \lambda_{1}(B_{\mathbb{N}^{m-1}(\kappa)}(r))$ we have that $f$ is an eigenfunction and  the expression  (\ref{eqIgualdade}) is zero (at each point of $\OM$).  This happens if and only if \[1=\sum_{i=1}^{m}\alpha_{i}^{2} = \sum_{i=1}^{m}\beta_{i}^{2}.\] On the other hand, we can write at each point $x\in \OM$ \[ \grad_{N}=\sum_{i=1}^{m}\alpha_{i} e_{i} + (grad_{N})^{\perp},\]where $(\grad_{N})^{\perp}$ is normal to the tangent space of $T_{x}\OM$. Likewise we can write
   \[\partial/\partial s=\sum_{i=1}^{m}\beta_{i} e_{i} + (\partial/\partial s)^{\perp}.\]Since $\Vert grad_{N}\Vert^{2} =\sum_{i=1}^{m}\alpha_{i}^{2}+\Vert (\grad_{N})^{\perp} \Vert$ and $\Vert \partial/\partial s \Vert^{2}= \sum_{i=1}^{m}\beta_{i}^{2}+\Vert (\partial/\partial s)^{\perp}\Vert^{2} $ we conclude that $(grad_{N})^{\perp}=0=(\partial/\partial s)^{\perp}$. Thus the tangent space $T_{x}\OM$ contains the vectors $\grad \rho_{N}$ and $\partial/\partial s$ for each $x\in \OM$. Thus, we could have   chosen in (\ref{eqbase}) an orthonormal basis for  $T_{x}\OM$  in the following way $e_{1}=\grad \rho_{N}$, $e_{2}=\partial /\partial s$ and $\{e_{3}, \ldots, e_{m}\}\subset\{ \partial /\partial \theta_{1}, \ldots, \partial /\partial \theta_{n-1}\}$. Clearly the set of vectors $\grad_{N}$ and $\partial /\partial s$ form  smooth vector fields on $\OM$ since they are the restrictions of smooth vector fields on $N\times \mathbb{R}$ to a smooth immersed submanifold. The integral curves of the vector field $\partial/\partial s$ in $\OM$ are $\{x\}\times \mathbb{R}$ contained in $\varphi(M)$ and $\OM=\varphi^{-1}(B_{N}(r)\times\mathbb{R})$ is not bounded.
  This proves Theorem (\ref{thm2}).
  \end{proof}

\subsection{Proof of Theorem \ref{thm3}}Let $\varphi :M\hookrightarrow N\times \mathbb{R}$ be a complete immersed $m$-submanifold with locally bounded mean curvature, where $N$ has  radial sectional curvature bounded above $K_{N}\leq \kappa$ along the geodesics issuing from $x_{0}$. Define $\tilde{\rho_{N}}:N\times \mathbb{R}\to \mathbb{R}$ by $\tilde{\rho_{N}}(x,t)=\rho_{N}(x)$,
$\rho_{N}(x)={\rm dist}_{N}(x_{0},x)$. Let
 $\OM(r)=\varphi^{-1}\left( B_{N}(x_{0},r)\times \mathbb{R}\right)$, $f=\tilde{\rho_{N}}\circ \varphi :\OM(r)\to  \mathbb{R}$ and $X=\grad f$.
 The idea is to choose $r<\min\{\inj_{N}(x_{0}), \pi/2\sqrt{\kappa}\}$, $\pi/2\sqrt{\kappa}=\infty$ if $\kappa \leq 0$, properly
  such that $\inf_{\OM(r)}\diver X>0$ then by Theorem (\ref{thmBM2}) we have that \[\lambda^{\ast}(\OM(r))\geq \left( \displaystyle\frac{\inf \diver X}{2 \sup \vert X\vert}\right)^{2}\cdot\] Observe that $\diver X=\triangle_{M}f$ and as in
(\ref{eqBF3}) we have \[\triangle_{M}f(x)=\left[\sum_{i=1}^{m}\hess_{N\times \mathbb{R}}\tilde{\rho_{N}}(e_{i},e_{i})+\langle \grad_{N\times \mathbb{R}}\tilde{\rho_{N}}, \stackrel{\rightarrow}{H} \rangle \right](\varphi (x))\]Where  $\stackrel{\rightarrow}{H}=\sum_{i=1}^{m}\alpha (e_{i},e_{i})$ is the mean curvature vector of $\varphi (M)$ at $\varphi (x)$ and $\{e_{1}, \ldots, e_{m}\}$ is  an orthonormal basis of $T_{x}M$   as in (\ref{eqBase}) identified with $\{d\varphi\cdot e_{1}, \dots, d\varphi\cdot e_{m}\}$. Now \begin{eqnarray}\sum_{i=1}^{m}\hess_{N\times \mathbb{R}}\tilde{\rho_{N}}(e_{i},e_{i})&=&\sum_{i=1}^{m}\hess_{N}\rho_{N}(e_{i},e_{i}) \nonumber\\
&=& \sum_{i=1}^{m}\sum_{j=1}^{n-1}(c_{j}^{i})^{2}\hess_{N}\rho_{N}(\partial/\partial\theta_{j},\partial/\partial\theta_{j})\\
&\geq & \sum_{i=1}^{m}(1-a_{i}^{2}-b_{i}^{2})\frac{C_{\kappa}}{S_{\kappa}}(r)\nonumber\label{eq3.8}
\end{eqnarray}On the other hand $\langle \grad_{N\times \mathbb{R}}\tilde{\rho_{N}},\stackrel{\rightarrow}{H}\rangle =\langle \grad_{N}\rho_{N},\stackrel{\rightarrow}{H}   \rangle$

 \begin{eqnarray}\label{eq3.9}\langle \grad_{N}\rho_{N},\stackrel{\rightarrow}{H}   \rangle & =&\langle (\grad_{N}\rho_{N})^{\perp},\stackrel{\rightarrow}{H}   \rangle\nonumber \\ &\leq& \vert H \vert\sqrt{1-\sum_{i=1}^{m}a_{i}^{2}}\\
&\leq& h(x_{0}, r)\sqrt{1-\sum_{i=1}^{m}a_{i}^{2}}\nonumber\end{eqnarray} Since $\vert \grad_{N}\rho_{N})^{\perp} \vert^{2} =(1-\sum_{i=1}^{m}a_{i}^{2}).$ Therefore from (\ref{eq3.8}) and (\ref{eq3.9}) we have
\[  \triangle_{M}f(x)\geq  (m-2)\frac{C_{\kappa}}{S_{\kappa}}(r)-h(x_{0}, r)>0\] We have two cases to consider.
 First is the case that $\vert h(x_{0},r)\vert <\Lambda^{2}<\infty$ and  we choose
$r\leq \min\{\inj_{N}(x_{0}), \pi/2\sqrt{\kappa}, (C_{\kappa}/S_{\kappa})^{-1}(\Lambda^{2}/(m-2))\}. $
 In case that
 $\lim_{r\to \infty}h(x_{0},r)=\infty$ there is $r_{0}$ so that $(m-2)\frac{C_{\kappa}}{S_{\kappa}}(r_{0})-h(x_{0}, r_{0})=0$ since we can assume without loss of generality that $h(x_{0},r)$ is a continuous non-decreasing function in $r$.
Then we choose
  $r\leq \min\{\inj_{N}(x_{0}), \pi/2\sqrt{\kappa}, (C_{\kappa}/S_{\kappa})^{-1}(h(x_{0},r_{0})/(m-2))\}. $ In both cases we have
   \[\lambda^{\ast}(\OM(r))\geq \left[\frac{(m-2)\frac{C_{\kappa}}{S_{\kappa}}(r)-h(x_{0}, r)}{2}\right]^{2}\cdot\] To prove Corollary (\ref{cor1}) just  see that $\OM (R_{p_{1}(M)})=M$ and $\lambda^{\ast}(M)=0$.

\bibliographystyle{amsplain}

\begin{thebibliography}{10}

\bibitem{barta}J. Barta,  {\em Sur la vibration fundamentale
d'une membrane.} C. R. Acad. Sci. \textbf{204}, (1937), 472--473.

\bibitem{berard} B\'{e}rard, Pierre H. {\em Spectral geometry: direct and inverse problems}.  Lect. Notes in Math. 1207, 1986,  Springer-Verlag.

\bibitem{berger-gauduchon-mazet} Berger, M., Gauduchon, P. and Mazet,
E.:  {\em Le Spectre d'une Vari\'{e}t\'{e} Riemannienes}. Lect.
Notes Math. 194, 1974, Springer-Verlag.

\bibitem{bessa-montenegro1} G. P. Bessa \and J. F. Montenegro,  {\em Eigenvalue estimates for
submanifolds  with locally bounded mean  curvature. } Ann. Global
Anal. and Geom., \textbf{24},  (2003),  279--290.


\bibitem{bessa-montenegro2} G. P. Bessa, \and J. Fabio Montenegro, {\em An Extension of Barta's Theorem and Geometric Applications}.  Ann. Global Anal. Geom. \textbf{31}  (2007),  no. 4, 345--362.

\bibitem{bessa-montenegro3} G. P. Bessa, \and J. Fabio Montenegro, {\em On compact $H$-hypersurfaces of $N\times\Bbb R$.}  Geom. Dedicata  \textbf{127},  (2007), 1--5.

\bibitem{candel} A. Candel, {\em Eigenvalue estimates for minimal surfaces in Hyperbolic space.} Trans. Amer. Math. Soc. \textbf{359}, (2007), 3567--3575.
 \bibitem{chavel}I.  Chavel,  {\em Eigenvalues in Riemannian
Geometry.}  Pure and Applied Mathematics, 1984, Academic Press,
INC.


\bibitem{cheng-li-yau}S. Y. Cheng,  P. Li, S. T. Yau,  {\em
Heat equations on minimal submanifolds and their applications.}
Amer. J. Math. \textbf{106},  1033--1065, (1984).

\bibitem{cheung-leung}  Leung-Fu Cheung and   Pui-Fai Leung, {\em Eigenvalue estimates for submanifolds with bounded
 mean curvature in the hyperbolic space.} Math. Z.  \textbf{236},   (2001), 525--530.


\bibitem{jorge-koutrofiotis}L.  Jorge,  and  D. Koutrofiotis, {\em An estimate for the curvature of
bounded submanifolds.} Amer. J. Math.,  \textbf{103},  (1980), 711--725.

\bibitem{jorge-xavier} L. Jorge, and F. Xavier, {\em A complete minimal surface in $R\sp{3}$ between two parallel planes.}  Ann. of Math. (2)  112  (1980), no. 1, 203--206.

\bibitem{pacomartin}F. Mart\'{\i}n, S. Morales, \textit{A complete bounded minimal cylinder in $\mathbb R\sp 3$. } Michigan Math. J.  \textbf{47}  (2000),  no. 3, 499--514.

\bibitem{meeks-rosenberg} W. Meeks and H. Rosenberg, {\em The theory
of minimal surfaces in $M^{2}\times \mathbb{R}$}.  Comment. Math. Helv.  \textbf{80}  (2005),  no. 4, 811--858.

\bibitem{meeks-rosenberg2} W. Meeks and H. Rosenberg, {\em Stable minimal surfaces in $M^{2}\times \mathbb{R}$}.
 J. Differential Geom.  \textbf{68}  (2004),  no. 3, 515--534.



 \bibitem{Nadirashvili}N.  Nadirashvili, {\em Hadamard's and
Calabi-Yau's conjectures on negatively curved and minimal surfaces.}
Invent. Math., \textbf{ 126},  (1996),  457--465.








\bibitem{kn:s-y} R. Schoen, and S. T. Yau, {\em Lectures on Differential Geometry.}
Conference Proceedings and Lecture Notes in Geometry and Topology, {\bf vol. 1}, (1994).


\end{thebibliography}

\end{document}